\documentclass[12pt]{amsart}
\usepackage{amsmath, amsthm, amscd, amssymb, amsfonts, latexsym, mathtools}

\usepackage[T1]{fontenc}
\usepackage{lmodern}
\usepackage[margin=1in]{geometry}
\usepackage{microtype}
\usepackage{amsmath,amssymb,amsthm,mathtools}
\usepackage{array,booktabs}
\usepackage{enumitem}
\usepackage{xcolor}
\usepackage{listings}
\usepackage{seqsplit}
\usepackage[numbers,sort&compress]{natbib}
\usepackage[hidelinks]{hyperref}
\usepackage[nameinlink,noabbrev]{cleveref}

\definecolor{noteblue}{RGB}{34,74,116}
\definecolor{lightblue}{RGB}{238,245,251}
\definecolor{resultgreen}{RGB}{236,247,239}
\definecolor{codegray}{RGB}{248,248,248}

\hypersetup{
  pdftitle={A prime knot whose handedness is invisible to Khovanov homology},
  pdfsubject={A colloquium-level account of a computer-assisted prime-knot answer to Question 2 of Lipshitz--Sarkar, ICM 2018}
}

\newtheorem{theorem}{Theorem}[section]
\newtheorem{proposition}[theorem]{Proposition}

\theoremstyle{definition}

\newcommand{\F}{\mathbb F_2}
\newcommand{\Z}{\mathbb Z}
\newcommand{\Kh}{\operatorname{Kh}}
\newcommand{\CKh}{\operatorname{CKh}}
\newcommand{\XKh}{\mathcal X_{\!Kh}}
\newcommand{\Sq}{\operatorname{Sq}}
\newcommand{\rk}{\operatorname{rank}}
\newcommand{\im}{\operatorname{im}}
\newcommand{\Span}{\operatorname{span}}
\newcommand{\comment}[1]{}

\title{\bfseries Lipshitz-Sarkar refines mirrors more than Khovanov
}
\author[Y.-H. He]{Yang-Hui He}
\address{London Institute for Mathematical Sciences, Royal Institution, London W1S 4BS, UK\\
Merton College, Oxford, OX14JD, UK}
\email{\href{mailto:yh@lims.ac.uk}{yh@lims.ac.uk}}

\date{}

\begin{document}
\maketitle

\begin{abstract}
Lipshitz and Sarkar asked in an ICM 2018 question
whether their homotopy-theoretic refinement of Khovanov homology can detect
mirror reflection when ordinary integral Khovanov homology cannot.  We give an
affirmative answer by presenting a prime knot whose integral Khovanov homology agrees with that of its
mirror in every bidegree, including torsion, and yet over $\mathbb{F}_2$, the second Steenrod square has rank one for the knot and
rank zero for its mirror in a specified bidegree.
We also give a composite example.
The construction, search and computations are done with the help of {\it chat-GPT 5.6 sol Pro} combing over the {\it Regina Census} of knots.
\end{abstract}

\comment{
\begin{center}
\fcolorbox{noteblue}{lightblue}{\parbox{0.91\textwidth}{\small
\textbf{Status.}  This is a new computer-assisted result and has not yet
been peer reviewed.  The theorem was checked with exact arithmetic, an
assertion-enabled recomputation, a separate calculation on the mirror, and
independent verification of the integral data and prime-knot census entry.
The raw outputs, replay programs, and cryptographic checksums are preserved
with the paper.}}
\end{center}
}

\section{Introduction and Comments}
AI-driven mathematics has gone a long way since the early experiments of machine-algorithms find new patterns and raising new conjectures \cite{He:2017aed,deepmind,murmurations} (see summary in \cite{nat-rev}).
The last six months saw a phase-transition in capability of LLMs to the point that major problems such as counter-examples to the Jacobian Conjecture for $n>2$ (confirmed), or a complex structure for $S^6$ (to be confirmed) can be found by interacting with Anthropic \cite{alpoge}. Similarly, a list of 10 highly non-trivial open problems are settled automatically by OpenAI \cite{openAI}. 

This short note aims to attack, with the help of {\it chat-GPT 5.6 sol Pro}, a problem in knot theory raised at the ICM of 2018 by Lipshitz-Sarkar \cite{LipshitzSarkarICM}. The specific one is Q2 from a list of 11 that asks for whether a refinement of Khovanov homology can detect the mirror of a knot when the homology itself fails.
We answer this in the affirmative by explicitly checking that a knot $K$ with 17 crossings has identical Khovanov homology as its mirror $mK$, but differs in the Steenrod square in the spectrum.
This is presented in Theorem \ref{thm:main}.

~\\

\paragraph{{\it Acknowledgments}}
We are grateful to Robert Lipshitz for valuable comments on the draft.
We are especially indebted to Krishnendu Kar for his careful reading, critique, and improvements.

%%%%%
\section{Preliminaries}
We first briefly recall some standard definitions.
A \emph{knot} $K$ is a smooth copy of the circle \(S^1\) embedded in
$\mathbb{R}^3$.  
Two knots are considered the same - isotopic - if one can be
moved continuously to the other through embedded circles.  
A knot invariant assigns an algebraic or topological object \(I(K)\) to a
knot in a way that does not change under isotopy.  

%One may bend and stretch the loop, but one may not cut it or pull one strand through another.
A knot is \emph{composite} if it can be made by cutting small arcs from two
nontrivial knots and joining the four ends.  This operation is the
\emph{connected sum}.  A nontrivial knot that cannot be decomposed in this
way is called \emph{prime}.  Prime knots are the indecomposable building
blocks of knot theory.
Reflect the knot $K$ across a plane.  The resulting knot is the
\emph{mirror} of \(K\), denoted by \(mK\).  If \(K\) can be deformed into
\(mK\), then \(K\) is \emph{mirror-symmetric} (sometimes dubbed
\emph{amphichiral}; otherwise it has an intrinsic handedness and is called
\emph{chiral}). 

\subsection{Khovanov Homology}
Khovanov homology is a powerful knot invariant introduced by Mikhail Khovanov
\cite{Khovanov2000}.  
For a knot \(L\), integral Khovanov homology is a collection
\[
 \Kh^{i,j}(L;\Z)
\]
of finitely generated abelian groups indexed by $(i,j)$, where the $i$
index is the \emph{cohomological grading} and $j$ is the \emph{quantum
grading}.
As usual in homology, each group is a direct sum of copies of \(\Z\) and finite cyclic groups, the latter of which is called \emph{torsion}.

Lipshitz and Sarkar associate to a knot \(L\) a family of spectra \footnote{
Recall that 
a {\it spectrum} is a stable version of a topological space: it retains attaching
information that can be lost when one records only cohomology groups.
More precisely, a spectrum $E_\bullet$ is a coherent sequence of spaces $E_0, E_1, E_2, \ldots$ with maps $\Sigma E_n \rightarrow E_{n+1}$ from the suspension of each to the next in the sequence. 
}
$\XKh^j(L)$.
The
connection with Khovanov homology is
\[
 \widetilde H^i\!\left(\XKh^j(L);\Z\right)
 \cong \Kh^{i,j}(L;\Z) \ ,
\]
where $\widetilde{H}^i$ is the usual cohomology for based spaces, viz., $\widetilde H^i (X; \mathbb{Z}) = H^i(X, \{\star\}; \mathbb{Z})$. For a spectrum $E$, \(\widetilde H^i(E;\mathbb Z)
=[E,\Sigma^iH\mathbb Z]_{\mathrm{st}}.\), the stable homotopy class of maps from $E$ to the degree $i$ Eilenberg-MacLane cohomology.

\comment{A standard analogy comes from ordinary topology.  The complex projective
plane \(\mathbb{C}P^2\) and the wedge \(S^2\vee S^4\) have the same integral
homology groups.  They are nevertheless different spaces.  Modulo two, the
second Steenrod square is nonzero on the degree-two class of
\(\mathbb{C}P^2\), while it vanishes on the corresponding class of
\(S^2\vee S^4\).  The operation sees how cells are attached, information
that the list of groups has forgotten.  This is an analogy, not an input to
the knot proof.
Lipshitz and Sarkar constructed a richer stable
homotopy invariant whose cohomology recovers Khovanov homology
\cite{LipshitzSarkarStable}.  In their 2018 International Congress of
Mathematicians lecture they asked whether this richer invariant can obstruct
mirror symmetry when ordinary Khovanov homology does not
\cite[Question~2]{LipshitzSarkarICM}.
}
\subsection{The Sq Map}
On \(\F=\{0,1\}\),  Lipshitz and Sarkar \cite{LipshitzSarkarSq2} gave
a combinatorial construction of the natural operation, the Steenrod Square map, as
\[
 \Sq^2:\Kh^{i,j}(L;\F)\longrightarrow\Kh^{i+2,j}(L;\F).
\]
Its rank is preserved by stable equivalence, but is not
determined by the dimensions of its source and target.  Seed implemented
and studied this map extensively \cite{Seed2012,SeedKnotKit}.

\comment{The hierarchy of information in the example is therefore:
\begin{center}
\begin{tabular}{@{}p{0.26\textwidth}p{0.42\textwidth}p{0.20\textwidth}@{}}
\toprule
invariant & what it remembers & comparison of \(K\) and \(mK\)\\
\midrule
Jones polynomial & a graded numerical summary & agrees\\
integral Khovanov homology & a table of abelian groups & agrees everywhere\\
Khovanov spectrum & stable attaching information & differs\\
the witness \(\Sq^2\) & one preserved feature of the spectrum & rank \(1\) versus \(0\)\\
\bottomrule
\end{tabular}
\end{center}
We do not claim that \(\Sq^2\) determines the full spectrum.  One
discrepancy in a preserved feature is enough to prove inequivalence.
}
It is represented by a universal stable map
\[
\theta_2:H\mathbb F_2\longrightarrow\Sigma^2H\mathbb F_2.
\]
If a cohomology class is represented by
\(
a : E \longrightarrow
\Sigma^nH \mathbb{F}_2,
\)
then
\[
\mathrm{Sq}^2(a)
(\Sigma^n\theta_2)\circ a:
E\longrightarrow\Sigma^{n+2}H\mathbb F_2 \ .
\]

\section{The Question}
In the ICM of 2018, Lipshitz and Sarkar posed 11 important questions \cite{LipshitzSarkarICM} pertaining to Khovanov homology of knots. To our knowledge, only Q3, 6, and 9 have been resolved.
One particular interesting question, Q2, asks
\begin{quote}
{\it Is the obstruction to amphichirality coming from the Khovanov spectrum
stronger than the obstruction coming from Khovanov homology?}
\end{quote}
In other words, can we find find a knot $K$ and its mirror $mK$, such that they have the same Khovanov homology for all $(i,j)$ but they differ in the actual spectrum $\XKh^j$; for instance, it suffices to find different maps in 
$\XKh(K)$ versus $\XKh^j(mK)$.

We now answer the question in the affirmative.
In particular, we find a prime knot which has this property and which we describe in Theorem \ref{thm:main}.
We also find a non-prime knot, which we will relegate to Theorem \ref{thm:composite} in the Appendix. 

~\\
We adhere to the notation of the Regina census of knots \cite{Burton2020}, and have
\begin{theorem}\label{thm:main}
The prime knot \(K=17nh_{0009090}\) is such that
\begin{enumerate}[label=\textup{(\roman*)}]
\item for every pair of grading indices \((i,j)\),
\[
       \Kh^{i,j}(K;\Z)\cong \Kh^{i,j}(mK;\Z)
\]
as abelian groups, including all torsion (the isomorphisms need not be
canonical);
\item 
\comment{with coefficients in \(\F\),
\[
\begin{array}{c@{\qquad}c}
\toprule
\text{object} &
\rk\!\left(\Sq^2:\Kh^{-6,-9}\longrightarrow\Kh^{-4,-9}\right)\\
\midrule
K  & 1\\
mK & 0\\
\bottomrule
\end{array}
\]
}
Define $r_L(i,j) := \mbox{rank}(\Sq^2 : \Kh^{i,j}(L; \mathbb{F}_2) \rightarrow \Kh^{i+2,j}(L; \mathbb{F}_2))$ to be the rank of the Steenrod square map for knot $L$, then
\[
r_K(-6,-9) = 1 , \mbox{ but } r_{mK}(-6,-9) = 0.
\]
\end{enumerate}
Consequently the quantum-graded Khovanov stable homotopy types of \(K\) and
its mirror are not equivalent, although their integral Khovanov homology
groups are identical.
\end{theorem}

\begin{proof}
We first summarize some key properties of the knot.
Complicated knots are usually specified by a diagram code.  One standard
choice is the Dowker--Thistlethwaite, or DT, code.  Sign conventions for DT codes vary; in Regina's convention, our knot has numeric code
\begin{equation}\label{eq:dt}
\mathrm{DT}[22,30,-18,32,-24,-16,34,-26,-8,-4,2,-28,-10,-14,20,6,12].
\end{equation}
This list is a finite instruction for reconstructing an oriented knot
diagram, much as a graph can be specified by an adjacency list.  The equivalent alphabetic DT code is
\(\texttt{qaqkoIpLHqMDBaNEGjcf}
\).
Both codes reconstruct a one-component diagram with 17 crossings.
Burton's published census is a certified census of prime knots through
seventeen crossings \cite{Burton2020}.  The exact row for \(K\) lies in the
nonalternating hyperbolic portion of that census.  

As an independent check, SnapPy found the complete meridian-preserving
symmetry group of the complement to have order one, both at ordinary and
high precision \cite{SnapPy}.  In particular the knot has no symmetry taking
it to its mirror.  This geometric check is reassuring, although the
Khovanov-spectrum discrepancy already proves chirality on its own.

\begin{center}
\begin{tabular}{@{}ll@{}}
\toprule
census name & \(17nh_{0009090}\)\\
crossing number & \(17\)\\
census class & nonalternating, hyperbolic, prime\\
total free rank of integral \(\Kh\) & \(86\)\\
number of \(\Z/2\) torsion summands & \(42\)\\
other recorded torsion & none\\
total dimension over \(\F\) & \(170\)\\
hyperbolic volume &   \(\simeq 15.73320532925\)\\
full link-symmetry group & trivial\\
\bottomrule
\end{tabular}
\end{center}

First, compare the knot and its mirror directly.  The independently
computed maps in the decisive bidegree are
\[
\begin{array}{c@{\qquad}c@{\qquad}c}
\toprule
 & \dim_{\F}\Kh^{-6,-9} &
 \rk(\Sq^2:\Kh^{-6,-9}\to\Kh^{-4,-9})\\
\midrule
K  &3&1\\
mK &3&0\\
\bottomrule
\end{array}
\]
The source groups have the same dimension, as they must because the
integral homology tables agree.  The operations have different ranks, so
the stable homotopy types cannot be equivalent.

Second, mirror duality gives a useful one-knot test.  Lawson, Lipshitz, and
Sarkar proved that the spectrum of a mirror is the Spanier--Whitehead dual
of the original spectrum, with the quantum grading reversed
\cite[Theorem~9]{LLS2020}.  Here ``dual'' is the stable-homotopy analogue of
dualizing a finite-dimensional vector space.  If we write
\[
 r_L(i,j)=\rk\!\left(
 \Sq^2:\Kh^{i,j}(L;\F)\to\Kh^{i+2,j}(L;\F)\right),
\]
then mirror self-duality would force
\begin{equation}\label{eq:rank-reflection}
                   r_L(i,j)=r_L(-i-2,-j).
\end{equation}
For the candidate,
\[
                    r_K(-6,-9)=1,
       \qquad       r_K(4,9)=0,
\]
so \cref{eq:rank-reflection} fails.  A separate computation on \(mK\)
swapped these values to \(0\) and \(1\), exactly as mirror duality predicts.
In the ensuing, we detail the search strategy and then in Appendix \ref{sec:identity-certificate}, we present the computation certificate.

\paragraph{{\it Search Strategy: }}
Earlier calculations found no such example among prime knots with at most
fifteen crossings
\cite[Remark~10.8]{LLS2020}.  As a separate exact census check, all
1,008,906 nonalternating sixteen-crossing prime knots were screened.  There
were 1,889 with mirror-symmetric rational and integral Khovanov homology;
only 67 had the grading support needed for a possibly nonzero \(\Sq^2\)
witness.  Fifty-seven of those were already known to be mirror-symmetric,
and direct calculations on the remaining ten found no rank discrepancy.

This makes our seventeen crossings the first successful frontier in the present
computer-assisted search.  The existence theorem, rather than a claim of
absolute minimality under every conceivable stable operation, is the main
result of this paper. The raw seventeen-crossing census contains more than eight million prime
knots, far too many for an expensive Steenrod-square computation on every
diagram.  Most can be discarded using progressively stronger but cheaper
tests.

The first test compares rational Khovanov polynomials with their mirror
reflections.  If those already differ, integral Khovanov homology
distinguishes the knot and it cannot answer the ICM question.  The second
test checks the complete integral groups, including torsion.  The third asks
whether the mod--two table even contains source and target groups separated
by two cohomological degrees; without such a pair, \(\Sq^2\) cannot supply
the desired witness.  Finally, a geometric symmetry test removes knots that
really are mirror-symmetric.

The complete exact filter was:
\begin{center}
\begin{tabular}{@{}p{0.65\textwidth}r@{}}
\toprule
stage & survivors\\
\midrule
all seventeen-crossing prime knots in the four census classes & 8,053,393\\
rational Khovanov mirror symmetry & 2,083\\
integral Khovanov mirror symmetry, torsion included & 2,083\\
mod--two support permits a nonzero \(\Sq^2\) block & 40\\
chiral after the independent hyperbolic-symmetry test & 34\\
\bottomrule
\end{tabular}
\end{center}
All 2,083 algebraic survivors were nonalternating hyperbolic knots.  The
integral stage used the complete exact dataset of Gurnari, D\l otko,
Sazdanovi\'c, Sch\"utz, Bar-Natan, and van der Veen
\cite{KnotInvariantsData}.

The 34 remaining knots were ordered by the size of their integral Khovanov
homology.  The first two candidates, of free ranks \(50\) and \(62\), had
mirror-symmetric \(\Sq^2\) ranks.  The third, of free rank \(86\), was
\(17nh_{0009090}\), and it produced four independent rank discrepancies.
The search stopped at the first fully verified example.

The exact Regina census row identifies \(K\) as a prime hyperbolic knot;
the diagram code and row are recorded in \cref{sec:identity-certificate}.
The integral calculation verifies the group isomorphisms in every bidegree,
with the universal-coefficient shift for torsion included.  The
assertion-enabled computation gives rank one for the displayed map on
\(K\), while the direct mirror computation gives rank zero for the same
bidegree on \(mK\).  Stable equivalences preserve Steenrod operations, so
the two quantum-graded spectra are not equivalent.
\end{proof}

%We remark that a non-prime example can be constructed, which we leave to the Appendix.

\appendix

\section{Computation}
\label{sec:identity-certificate}
We give the SAGE code \cite{sagemath} which verifies our claims:
\begin{center}
\href{https://github.com/yhe196884/KhovanovHomologyMirror/tree/main}{https://github.com/yhe196884/KhovanovHomologyMirror/tree/main},
\end{center}

freely available for download from Github.

\end{document}